\documentclass[reqno]{amsart}
\usepackage[latin1]{inputenc}
\usepackage[T1]{fontenc}
\usepackage{amssymb, amsmath, color, textcomp, url,tikz, mathtools}

\usetikzlibrary{matrix,arrows}
\usetikzlibrary{fit}
\usetikzlibrary{positioning}
\usepackage[labelformat=empty]{caption}
\usepackage{mdwlist}
\usepackage{etoolbox}

\RequirePackage{ifpdf}
\ifpdf
   \usepackage[pdftex]{hyperref}
\else
   \usepackage[hypertex]{hyperref}
\fi

\usepackage{soul}
\usepackage{enumerate,xspace}
\newtheorem*{teoA}{Theorem A}
\newtheorem*{teoB}{Theorem B}
\newtheorem*{teoC}{Theorem C}
\newtheorem*{teoD}{Theorem D}

\theoremstyle{plain}
\newtheorem{theorem}{Theorem}[section]
\newtheorem{cor}[theorem]{Corollary}
\newtheorem{prop}[theorem]{Proposition}
\newtheorem{lemma}[theorem]{Lemma}

\newcounter{claimCount}
\AtBeginEnvironment{proof}{\setcounter{claimCount}{0}}

\theoremstyle{definition}
\newtheorem{remark}[theorem]{Remark}
\newtheorem{fact}[theorem]{Fact}
\newtheorem{definition}[theorem]{Definition}
\newtheorem{example}[theorem]{Example}

\newtheorem*{notation}{Notation}

\newcommand{\nc}{\newcommand}

\nc{\Z}{\mathbb{Z}}
\nc{\Q}{\mathbb{Q}}
\nc{\N}{\mathbb{N}}
\nc{\F}{\mathbb{F}}
\nc{\UU}{\mathbb{U}}
\nc{\C}{\mathbb{C}}

\nc{\PR}{\mathrm{Pr}}
\nc{\SF}{\mathrm{SF}}

\nc{\M}{\mathcal{M}}
\nc\LL{\mathcal L}
\nc\II{\mathcal I}
\renewcommand{\t}{{\tt t}}

\nc{\stt}{\operatorname{St}}
\nc{\stab}{\operatorname{Stab}}
\nc{\GO}[1]{G_{#1}^{00}}
\nc{\Cos}[1]{\operatorname{Cos}(#1)}

\nc{\band}[1]{\bar d_{\mathcal{#1}}}
\nc\BD{\operatorname{BD}}

\nc{\dcl}{\operatorname{dcl}}
\nc{\dclq}{\operatorname{acl^\text{eq}}}

\nc{\acl}{\operatorname{acl}}
\nc{\aclq}{\operatorname{acl^\text{eq}}}
\nc{\nf}[1]{_{\mid {#1}}}
\nc{\restr}[1]{\xspace_{\upharpoonright {#1}}}

\nc{\sbgp}[1]{\langle\xspace {#1}\xspace\rangle}
\nc{\strp}[1]{\langle\xspace {#1}\xspace\rangle_\text{pure}}

\nc\CAN{\operatorname{CB}}
\nc\inv{ ^{-1}}

\nc{\tp}{\operatorname{tp}}
\nc\cb{\operatorname{Cb}}
\nc\U{\operatorname{U}}
\nc{\cf}{\text{cf.\,}}
\nc{\eg}{\text{e.g. }}

\def\Ind#1#2{#1\setbox0=\hbox{$#1x$}\kern\wd0\hbox to
	0pt{\hss$#1\mid$\hss} \lower.9\ht0\hbox to
	0pt{\hss$#1\smile$\hss}\kern\wd0}
\def\Notind#1#2{#1\setbox0=\hbox{$#1x$}\kern\wd0\hbox to
	0pt{\mathchardef\nn="0236\hss$#1\nn$\kern1.4\wd0\hss}\hbox to
	0pt{\hss$#1\mid$\hss}\lower.9\ht0 \hbox to
	0pt{\hss$#1\smile$\hss}\kern\wd0}
\def\ind{\mathop{\mathpalette\Ind{}}}
\def\nind{\mathop{\mathpalette\Notind{}}}

\def\indip{\mathop{\ \ \hbox to 0pt{\hss$\mid^{\hbox to
				0pt{$\scriptstyle P$\hss}}$\hss}
		\lower4pt\hbox to 0pt{\hss$\smile$\hss}\ \ }}
\def\nindip{\mathop{\ \ \hbox to 0pt{\hss$\!\not{\mid}^{\hbox to
				0pt{$\scriptstyle\, P$\hss}}$\hss}
		\lower4pt\hbox to 0pt{\hss$\smile$\hss}\ \ }}

\begin{document}

\title{Supersimplicity and arithmetic progressions}

\date{\today}

\author{ Amador Martin-Pizarro and Daniel Palac\'in}
\address{Mathematisches Institut,
  Albert-Ludwigs-Universit\"at Freiburg, D-79104 Freiburg, Germany}
\email{pizarro@math.uni-freiburg.de}

\address{Departamento de \'Algebra, Geometr\'ia y Topolog\'ia, 	
	Facultad de Ciencias Matem\'aticas, 
		Universidad Complutense de Madrid, Plaza Ciencias 3, 28040, 
		Madrid, Spain}
	\email{dpalacin@ucm.es}
 

\begin{abstract}
The main motivation for this article is to explore the connections between the existence of certain combinatorial patterns (as in van der Corputs's theorem on arithmetic progressions of length $3$) with well-known tools and theorems for definable groups in simple theories. 
 In the last sections of this article, we apply our model-theoretic results to bound the number of initial points starting few arithmetic progression of length $3$ in  the structure of the additive group of integers with a predicate for the prime integers, assuming Dickson's conjecture, or with a predicate for the square-free integers, as well as for asymptotic limits of finite fields. Our techniques yield similar results for the elements appearing as distances in skew-corners and for S\'ark\"ozy's theorem on the distance of distinct elements being perfect squares. 
\end{abstract}

\maketitle

\section*{Introduction}

Van der Corput \cite{jC39} proved that the set of primes contains infinitely many arithmetic progressions of length $3$ (in short $3$-AP's), in contrast to Roth's theorem \cite{kR53}, which is only valid for subsets of the natural numbers (or of the integers) of positive upper density. In \cite{bG05}, Green showed that every subset of the primes of positive relative upper density contains infinitely many $3$-AP's. The case of arithmetic progressions of higher length in the primes (akin to Szemer\'edi's theorem \cite{eS75}) was shown by Green and Tao in their celebrated theorem \cite{GT08}. Note that the associated complexity for finding a $k$-AP with $k\ge 4$ is much higher than solving a single $3$-AP, or equivalently, solving the equation $x+z=2y$ with all entries in the subset of interest. 

The work of Tao and Ziegler \cite[Corollary 1.6]{TZ23} yields with a minor  adaption, privately communicated, that the subset of the prime integers $\pm p$, with $p$ a prime number in $\N$, is \emph{unstable} in the model-theoretic sense (see also \cite[Proposition 3.6]{KS17}). Recall that a subset $X$ of $\Z$ is unstable if there is an infinite sequence $(a_n, b_n)_{n\in \N}$ of integers such that $a_n+b_m$ belongs to $X$ if and only if $n < m$. Stable subsets of a group have received consideration in the last years, by a fruitful adaptation of some of the ideas of \emph{local stability} developed by Hrushovski and Pillay \cite{HP94} in order to treat problems of combinatorial interests (see among others \cite{TW19, CPT20, gC21, MPPW21}, as well as \cite{CP24} in an analytic context). Despite the fact that the prime integers are not stable, Kaplan and Shelah \cite{KS17} showed that the theory of the structure $(\Z, 0, +, \PR)$ with a distinguished predicate $\PR$ for all prime integers is simple, assuming Dickson's conjecture. Simple theories, introduced by Shelah \cite{sS80}, provide a suitable framework to extend many of the techniques and tools from stability. A fundamental difference between stability and simplicity is that the latter cannot be characterised in terms of a local behaviour, which may explain why it may seem less suitable for possible applications of model theory to additive combinatorics. 

In \cite[Theorem 3.14]{MPP24}, several of the ideas in Hrushovski's \emph{stabilizer theorem} \cite{eH12} were adapted in order to provide a non-quantitative version of Roth's theorem for finite abelian groups of odd order. The main tool of the proof was an adaptation of a classical result due to Pillay, Scanlon and Wagner \cite{PSW98} on the behaviour of three \emph{coset-compatible} generic types in a group definable in a simple theory, see Fact \ref{F:pqr}. One of the main motivations for this article is to explore the connection between the existence of certain combinatorial patterns (as in van der Corputs's theorem) with the result of three authors above, under the global assumption that the structure we work with has a simple theory. 
During the CIRM conference in November 2024, we noticed that a similar interaction between the above two topics is also present in Pillay and Stonestrom's \emph{arithmetic regularity lemma} \cite{PS24} for finite fields, whose asymptotic theory is \emph{supersimple theory of rank $1$}. Recall that a first-order theory $T$ is supersimple theory of rank $1$ if in no model $M$ of $T$ there is a definable bipartite relation $R\subset M\times M^k$ and a sequence $(\bar b_n)_{n\in \N}$ in $M^k$ such that each fiber $R_{\bar b_n}$ is infinite yet any two distinct fibers $R_{\bar b_n}$ and $R_{\bar b_m}$ have finite intersection. 

Our main purely model-theoretic result is the following statement. 

\begin{teoA}\textup{(}\cf Theorem \ref{T:Main}\textup{)}~
Consider a definable subset $X$ of an abelian group $G$ without involutions, all definable in a supersimple theory of rank $1$.  The  subset 
\[ \left\{x \in X\ | \ \text{ there are only finitely many $g$ in $G$ with $x$, $x+ g$ and $x+2g$ in $X$}\right\} 
\] is definable and finite. In particular, if $X$ is infinite, all but finitely many elements $a$ in $X$ are initial points of infinitely many $3$-AP's in $X$. 
\end{teoA}

As mentioned before, assuming Dickson's conjecture, Kaplan and Shelah showed that the theory of the abelian group $(\Z, 0, +, \PR)$  is supersimple of rank $1$. A key step in their proof is a partial quantifier elimination result, so that the definable sets are exactly the \emph{constructible} sets (\cf Definition \ref{D:basic_subset}) given by finite unions of subsets of the form 
\[ \left\{ \bar z\in \Z^m \ | \  \Sigma(\bar z) \text{ holds with }  \bigwedge\nolimits_{i,k}  t_{i,k}(\bar z) \in \PR_k^{\epsilon_{i, k}}   \right\}, 
\] for a finite collection of affine-linear forms $t_{i, k}(\bar z)$ with integer coefficients, a finite system of congruences $\Sigma(\bar z)$ involving some of the above forms, as well as  some $\epsilon_{i,k}$'s in $\{-1,1\}$, with the convention that $\PR_k^{-1}=\Z\setminus \PR_k$ and $\PR_k=\{n\in k\Z \ | \ n/k \in \PR\}$. The \emph{complexity} of a constructible set is the minimal length of a formula defining it. 

In the particular structure $(\Z, 0, +, \PR)$, Theorem A above yields the following:
\begin{teoB}\textup{(}\cf Corollary \ref{C:Main}\textup{)}~
Assume Dickson's conjecture. For every $k\ge 1$ and every infinite constructible set $X$, all but finitely many elements in $X$ are initial points of infinitely many arithmetic progressions in $X$ of length $3$ with distance divisible by $k$.

More precisely, given a complexity $C>0$ and an integer $k\ge 1$, there are constants $M$ and $N$ in $\N$ such that for every constructible subset $X(x;\bar y)$ in the structure $(\Z, 0, +, \PR)$ of complexity at most $C$ and every choice $\bar b$ of parameters, the set 
\[ \left\{a \in X(x, \bar b) \ | \ \exists^{\le M} g  \text{ in  $k\Z$  with $a$, $a+ g$ and $a+2g$ in $X(x, \bar b)$}\right\} 
\]
is finite of size at most $N$. Moreover, if there are at least $M+1$ many $g$'s in $k\Z$ for the element $a$ in $X(x, \bar b)$, then there are infinitely many such $g$'s. 
\end{teoB}

The same statement holds for  $(\Z, 0, +, \SF)$ with a distinguished predicate $\SF$ for all square-free integers, without assuming Dickson's conjecture. At the moment of writing, even if the element $a$ belongs to $X(\N,\bar b)=X(x,\bar b)\cap \N$, we do not know whether there are infinitely many $g$'s in $\N$ (or any $g$ at all, for that matter!)  as in the above theorem. Our techniques cannot be applied directly, since $\N$ cannot be definable in any of the supersimple structures $(\Z, 0, +, \SF)$ or $(\Z, 0, +, \PR)$. Indeed,  the monoid structure on $\N$ canonically defines the linear order of $\N$, contradicting the elimination of $\exists^\infty$ (see Lemma \ref{L:Existsinfty}). 

As mentioned before, a classical class of structures whose theories are supersimple of rank $1$ are \emph{pseudo-finite fields}, that is, non-principal ultraproducts of finite fields, or rather fields elementarily equivalent to such ultraproducts, in the language of rings \cite{jA68,CDM92}. The core idea at the work of Chatzidakis, van den Dries and Macintyre \cite[Main Theorem]{CDM92} was to generalize the Lang-Weil estimates for varieties to arbitrary uniformly definable sets over finite fields. Their work has played an important role in Tao's algebraic arithmetic lemma \cite{tT15} (see also \cite{tTblog}), as well as in the recent work of Boege \cite{tB25} on the entropic behaviour of definable sets over finite fields. The Main Theorem of \cite{CDM92} roughly states the following: For every complexity $C>0$, there are finitely many values $(d, \mu)$ with $d$ in $\N$ and $\mu$ in $\Q^{>0}$ such that for every finite field $k$ and every non-empty definable subset $X\subset k^n$ of complexity at most $C$ there is a pair $(d, \mu)$ with $0\le d\le n$ and $$|X|=\mu |k|^d + O_C\big(|k|^{d-\frac{1}{2}}\big).$$ Moreover, the above estimate yields a dichotomy for the cardinality of definable sets within a uniformly definable family $(Y_{\bar b})$ of $k$ of complexity bounded by $C$, that is, there are constants $\ell_C$ in $\N$ and $0<\delta_C\le 1$ such that either $|Y_{\bar b}|\le \ell_C$ or its density $|Y_{\bar b}|/|k|$ is bounded from below by $\delta_C$. 

In this context, our Theorem A translates as follows:

\begin{teoC}\textup{(}\cf Theorem \ref{T:Roth_fte}\textup{)}~ 
Given a complexity $C>0$ and natural numbers $r$ and $s$, there is a constant $t=t(C, r, s)$ in $\N$ such that for every finite field $k$, every definable additive subgroup $H$ of $k$ of index at most $s$ and every definable subset $X$ of $k$, if both $X$ and $H$ have complexity at most $C$, then the set of \emph{bad} initial points
\[ 
\left\{x\in X\ | \ \exists^{\le r} g \text{ in $H$ with $x, x+g$ and $x+2g$ in $X$} \right\}
\] 
has size at most $t$.  Moreover, there are two constants $\eta=\eta(C,s)>0$ in $\N$ and $\delta=\delta(C)>0$ such that whenever $|X|\ge \eta$, the subset  
\[ 
\left\{x\in X\ | \ \exists^{\ge  \delta|k|} g  \text{ in $H$ with $x, x+g$ and $x+2g$ in $X$} \right\}
\] has cardinality at least $\delta |k|$. 
\end{teoC}

We conclude this introduction by noticing that similar techniques apply for other combinatorial patterns. As mentioned at the beginning, stability has played an important role in problems of additive combinatorics from a model-theoretic perspective. For instance, in previous joint work with Wolf \cite{MPPW24}, we studied the presence of {\em corners} and {\em squares} for a stable $2$-dimensional relation in a group. Our proof heavily relies on the \emph{stationarity principle} which characterizes stability. In simple (unstable) theories, stationarity does no longer hold, so the techniques from \cite{MPPW24} do not immediately apply. However, in the particular case of pseudo-finite fields of characteristic $0$, some of the model-theoretic methods can be easily adapted to yield the following result on \emph{skew-corners} for $2$-dimensional definable sets in finite fields (\cf \cite{kP24, JLO24, lM24}). 

\begin{teoD}\textup{(}\cf Theorem \ref{T:Skew_fte}\textup{)}~ 
Given a complexity $C>0$, there are constants $n_0$ and $t_C$ in $\N$ as well as  $\delta_C>0$ and $\delta'_C>0$ in $\Q$ such that for every finite field $k$ of characteristic at least $n_0$ and  every definable subset $S$ of $k^2$ of complexity at most $C$ with $|S|\ge \delta_C |k|^2$,  all but $t_C$ many points $g$ in $k$ satisfy that $|\mathcal X_g|\ge \delta'_C|k|^3 $, where $\mathcal X_g$ is the set of \emph{symmetric skew-corners} induced by $g$
in $S$: \[ \mathcal X_g =\{(x, y, y') \in k^3 \ | \ (x, y), (x, y'), (x+g, y) \text{ and } (x, y'+g) \text{ all lie in } S\}. \] 
\end{teoD}

\subsection*{Acknowledgments} The second author conducted
research supported by project STRANO PID2021-122752NB-I00 and Grupos UCM 910444. We would like to thank Anton Bernshteyn and Yifan Jing for their helpful comments and suggestions, which led us to improve a previous version of Theorem C. We would also like to thank the anonymous referee for the helpful comments and remarks. 

\section{(Super-)Simplicity}\label{S:SS}
All throughout this work, we will assume a familiarity with (geometric) model theory, basically the first five chapters of \cite{TZ12}. More advanced references for simplicity are \cite{eC11} and \cite{fW00}. For the sake of the presentation, we will attempt to provide a more self-contained exposition, choosing as definitions some equivalent formulations of existing notions instead of the classical ones. 

We consider a first-order complete theory $T$ in some fixed language $\LL$ and work in a sufficiently saturated (infinite) model $\mathbb U$ of $T$. All tuples and subsets are taken inside $\UU$, unless explicitly stated. In this section, finite tuples will be denoted by $a$, $b$, or $x$, $y$, and need not be singletons, unless stated so. We identify a formula $\varphi(x, b)$ with the underlying definable set 
\[ 
\varphi(\UU, b)=\{a \in \UU^{|x|} \ | \ \varphi(a,b) \text{ holds}\}.
\]

We first recall Shelah's definitions of dividing and forking \cite{sS90}.

\begin{definition}\label{D:forking}
A formula $\varphi(x,b)$ \emph{divides} over the subset $C$ of $\UU$ if there exists a $C$-indiscernible sequence $(b_n)_{n\in \N}$ with $b_0=b$ such that the set of formulae $\{\varphi(x, b_n)\}_{n\in \N}$ is inconsistent. 

A formula $\varphi(x, b)$ \emph{forks} over $C$ if it implies a finite disjunction of formulae $\psi_1(x, d_1), \ldots, \psi_n(x, d_n)$, with $d_1, \ldots, d_n$ in $\UU$, each of which divides over $C$ (with respect to different indiscernible sequences). 

Given $C\subset B$, a partial type $\pi(x)$ over $B$  closed under deductions \emph{forks}, resp. \emph{divides} over $C$, if it contains a formula which does. 
\end{definition}
The collection of definable sets which fork over a fixed based set $C$ of $\UU$ form an ideal, and thus every partial type $\pi(x)$ over $B$ as above which does not fork over $C$ admits an extension to every superset $D\supset B$ which does not fork over $C$. 

\begin{definition}\label{D:Simple}\textup{(}\cf \cite[Proposition 4.8 \& 4.13]{eC11}\textup{)}~
The theory $T$ is \emph{simple} if for every single variable $x$ and every type $p(x)$ over $B$ there is some subset $C\subset B$ such that $|C|\le |T|$ and $p$ does not divide over $C$. 
\end{definition}

The impressive result of Kim and Pillay \cite[Theorem 4.2]{KP97} yields that simplicity of $T$ is equivalent to the existence of a suitable notion of independence, satisfying certain axiomatic properties.
\begin{fact}\label{F:Simple_indep}\textup{(}\cf \cite[Theorem 4.2]{KP97}\textup{)}~
The theory $T$ is \emph{simple} if it admits a ternary relation $\ind$ among triples $(A, B, C)$ of subsets of $\UU$, denoted by $A\ind_C B$, which satisfies the following conditions:
\begin{description}
\item[Invariance] If $A, B, C\equiv A', B', C'$, then $A\ind_C B$ if and only if $A'\ind_{C'} B'$.\vskip1mm
 \item[Local character] For any finite $A$ and any $B$, there exists a subset $C\subset B$ with $|C|\le |T|$ such that $A\ind_{C} B$.\vskip1mm
\item[Finite character] We have that $A\ind_C B$ if and only if $a\ind_C b$ for all finite tuples $a$ in $A$ and $b$ in $B$.\vskip1mm
  \item[(Full) Extension] If $C\subset B$, given $A$, there is some  $A'\equiv_{C} A$ with $A'\ind_C B$. In particular, {\bf Existence} holds, that is $A\ind_B \acl(B)$. \vskip1mm
\item[Symmetry] If $A\ind_C B$, then $B\ind_C A$.\vskip1mm
 \item[Transitivity] We have that $A\ind_C B, D$ if and only if $A\ind_C B$ and $A\ind_{B,C} D$, where $X,Y$ denotes the set-theoretic union $X\cup Y$.\vskip1mm
\item[Independence theorem over models]  If $M$ is a model and $C\ind_M D$, given $A\equiv_M A_1$ with $A\ind_M C$ and $A_1\ind_M D$, then there exists some $A_2$ realizing $\tp(A/M,C)\cup \tp(A_1/M, D)$ with $A_2\ind_M C, D$.\vskip1mm
\end{description}
Moreover, if such a ternary relation exists, then it is unique and coincides with \emph{non-forking independence}, that is, the subset $A$ is \emph{independent} from $B$ over $C$, denote by $A\ind_C B$, if the type $\tp(A/B,C)$ (as a type in infinitely many variables) does not fork over $C$. 
\end{fact}

Furthermore, in a simple theory, the above independence relation satisfies these additional properties:
\begin{description}
\item[Algebraicity] We have that $A\ind_B A$ if and only if $A$ is algebraic over $B$. 
\item[Closure]  We have that $A\ind_C B$ if and only if $\acl(A, C)\ind_{C} \acl(B, C)$.  
\end{description}
Both properties hold easily for non-forking independence, but they can be abstractly deduced combining the properties in Definition \ref{D:Simple}.

For our purposes, we shall restrict our attention to a specific subclass of simple theories.

\begin{definition}\label{D:Rank1}
 The  theory $T$ is \emph{supersimple of rank $1$} if every formula $\varphi(x, b)$ with $x$ a single variable dividing (or equivalently, forking) over $\emptyset$ needs to be algebraic. This is equivalent to the same condition for every formula dividing (or forking) over some arbitrary parameter set $C$.  
\end{definition}
Note that every supersimple theory $T$ of rank $1$ is simple, directly from Definition \ref{D:Simple}. 
\begin{remark}\label{R:Rank1_equiv}

A theory $T$ is supersimple of rank $1$ if and only if  for every singleton $a$ and subsets $C\subset B$ with $a$ non-algebraic over $C$, we have that $a\nind_C B$ if and only if $a$ belongs to $\acl(B)$. 

In particular, symmetry of non-forking independence yields that the algebraic closure operator induces a pregeometry (or matroid) satisfying Steinitz's exchange principle. 
\end{remark}
\begin{proof}
Notice that $a\nind_C B$ whenever $a$ lies in $\acl(B)\setminus \acl(C)$ always holds by {\em Closure} and Definition \ref{D:forking}.  If every formula $\varphi(x, b)$ forking over $C$ is algebraic, it is clear that every type $\tp(a/B)$ with $a\nind_C B$ must be algebraic as well. 

For the other direction, consider now a formula $\varphi(x, b)$ with $x$ a single variable dividing over $C$. Every type $\tp(a/B)$ containing this formula needs to be algebraic over $Cb$, so by compactness the formula $\varphi(x, b)$ is algebraic. 
\end{proof}

\begin{lemma}\label{L:Existsinfty}
    Every supersimple theory of rank $1$ \emph{eliminates the quantifier $\exists^\infty$}, that is, for every formula $\varphi(x, y)$ with $x$ a single variable, there exists some $k_\varphi$ in $\N$ such that for every $b$ in $\UU$, either $\varphi(\UU, b)$ is infinite or it has cardinality at most $k_\varphi$. In particular, the collection of $b$'s with $\varphi(\UU, b)$ finite forms a definable set. 
\end{lemma}
\begin{proof}
By \cite[Lemma 4.3]{eH03}, it suffices to show that our theory is an $S_1$-theory (\cf \cite[Definition 4.1]{eH03}), that is, there is no formula $\varphi(x, y)$ with $x$ a single variable and an indiscernible sequence $(b_n)_{n\in \N}$ with $\varphi(x, b_n)$ infinite but $\varphi(x, b_n)\land \varphi(x, b_{n+1})$ finite. Assume otherwise that there exists such a formula $\varphi$ and such an indiscernible sequence $(b_n)_{n\in\N}$, so the formula $\varphi(x, b_0)\land \neg \varphi(x, b_1)$ would clearly divide over $\emptyset$, witnessed by the indiscernible sequence of pairs $(b_n, b_{n+1})_{n\in\N}$. Now, the set 
\[ \varphi(x,b_0)=\left(\varphi(x, b_0)\land \neg\varphi(x, b_1)\right) \lor \left(\varphi(x, b_0)\land \varphi(x, b_1) \right) \]
would be finite, which gives the desired contradiction.  
\end{proof}
It is immediate that every $S_1$-theory is supersimple of rank $1$, so these two notions coincide.  Furthermore, using the previous lemma, a theory is supersimple of rank $1$ if and only if there is no formula $\varphi(x, y)$ with $x$ a single variable and a  sequence $(b_n)_{n\in \N}$ (possibly not indiscernible) with $\varphi(x, b_n)$ infinite but $\varphi(x, b_0)\land \varphi(x, b_1)$  finite. 

\begin{example}\label{E:s1_theories}~

\begin{enumerate}[(a)]
\item The theory of the random graph is axiomatized in the language of graphs by stating that for every two finite disjoint sets  $A$ and $B$ of vertices there is a vertex connected to every point in $A$ and not connected to any point in $B$. The probabilistic Erd\"os-R\'enyi countable graph is a model of this theory \cite[Section 4.3]{pC90}. A straightforward quantifier elimination result yields that this theory is supersimple of rank $1$, where the model theoretic algebraic closure of a set $X$ coincides with the set $X$ itself. 
\item Every strongly minimal theory, or more generally every stable theory of Lascar rank $1$, is supersimple of rank $1$. In particular, the theories of the abelian groups $\mathbb Q$ and of $\mathbb Z$  in the language of groups are supersimple of rank $1$. 
\item In the language $\LL_P$ of groups expanded by a unary predicate $P$, the theory of the abelian group $\mathbb Z$  where the predicate $P$ is interpreted as the set of square-free integers   is supersimple (yet unstable) of rank $1$ \cite{BT21}. The same is true, assuming Dickson's conjecture, whenever the predicate $P$ is interpreted as the set of prime integers \cite{KS17}. These two expansions will be discussed in more detail in Section \ref{S:integers}.
\item The common theory of non-principal ultraproducts of finite fields in the language of rings is supersimple of rank $1$ \cite[Proposition 4.5]{CDM92}. This particular theory will be discussed in more detail in Section \ref{S:psf}.
\end{enumerate}
\end{example}

Remark \ref{R:Rank1_equiv} and Lemma \ref{L:Existsinfty} yield every supersimple theory of rank $1$ is \emph{geometric} \cite[Definition 2.1 \& Examples on p. 316]{jG05}, so there is a well-behaved additive dimension function:  Given a finite tuple $a$ and a subset of parameters $B$, set $\dim(a/B)$ to be the length of a maximal subtuple $a'$ of $a$ which is $B$-independent, that is, for each coordinate $x$ of $a'$, we have that $x$ is not algebraic over $B, a'\setminus \{x\}$.  In particular, for every $C\subset B$ 
\[ a\ind_C B \ \text{ if and only if} \ \dim(a/C)=\dim(a/B).\] 
Transitivity of non-forking independence corresponds to the following equation for all parameter sets $C$ and all finite tuples $a$ and $b$:
\[
\dim(a,b/C)=\dim(a/C,b)+\dim(b/C).
\]
More generally, given a definable subset $X$ over the parameter $B$, set 
\[ 
\dim(X)=\max \{ \dim(a/B) \ | \ a \in X\}, 
\]
with the convention that $\dim(\emptyset)=0$.
\begin{remark}\label{R:dim}~
\begin{enumerate}[(a)]
    \item A definable subset $X$ is finite if and only if $\dim(X)= 0$. 
    \item For any two definable sets $X$ and $Y$ of $\UU^n$, \[ \dim(X\cup Y)=\max\{\dim(X),\dim(Y)\}.\]
    \item By Lemma \ref{L:Existsinfty}, given a formula $\varphi(x, y)$  and an integer $k$, the collection 
    \[  \left\{ b\in \UU^{|y|} \ | \ \dim\left(\varphi(x, b)\right)=k \right\} 
    \] is definable without parameters \cite[Fact 2.4]{jG05}.  
\end{enumerate}
\end{remark}

From now on, we will fix a supersimple theory $T$ of rank $1$ such that the universe $\UU$ admits a group law definable in $T$. For ease of the presentation, we will denote the universe by $G$ and assume that both the group law and the inverse map are definable without parameters, as in Example \ref{E:s1_theories} (b)--(d). 

In the particular case of a supersimple theory of rank $1$, the classical definitions of generic types becomes much simpler (no pun intended) to present. The reader should be aware that our definition of genericity below only applies to this particular case and does not correspond to the general definition \cite[Definition 4.1.1]{fW00}.

\begin{definition}\label{D:gen}
A unary type $p$ over the parameter set $A$ is \emph{generic} if the type $p$ is not algebraic. More generally, an  element $g$ of $G$ is \emph{generic} over the parameter set $A$ if $\tp(g/A)$ is not algebraic, equivalently if $\dim(g/A)=1$. 
\end{definition}
\begin{remark}\label{R:prop_gen}
It follows immediately from the definitions that a non-forking extension of a generic type remains generic. Moreover, if $g$ is generic over $A$, then so are $g\inv$, $a\cdot g$ and $g\cdot a$ for every $a$ in $\acl(A)$. In particular, the product of two generic independent elements over $A$ is again generic and independent of each factor over $A$. Every element $x$ of $G$ can be written as the product of two generic elements over $A$ (which need not be an independent pair over $A$, unless $x$ itself is generic over $A$). 
\end{remark}

\begin{definition}\label{D:G00}
Given a subset $A$ of $G$, we denote by $\GO A$ the  intersection of all subgroups type-definable over $A$ and of bounded index. This is the smallest type-definable 
subgroup over $A$ of bounded index, see \cite[Definition 4.1.10 \& Lemma 4.1.11]{fW00}. Furthermore, the subgroup $\GO A$ is normal in $G$. 
\end{definition}

\begin{fact}\label{F:G00}~
Fix a subset $A$ of $G$. 
\begin{enumerate}[(a)]
    \item The subgroup $\GO A$ is an intersection of $A$-definable subgroups of $G$ of finite index \cite[Theorem 5.5.4]{fW00}. 
    \item We say that  \emph{the connected component of $G$ exists (absolutely)} if $\GO \emptyset=\GO A$ for every subset of parameters $A$. For a supersimple expansion of rank $1$ of the additive group $\Z$, every definable subgroup of finite index is definable without parameters, so the connected component exists. If the abelian group $G$ is divisible, the connected component exists, since $G=\GO A$. 
    \item Given a generic type $p$ over $A$ and an element $g$ of $\GO A$ generic over $A$, the partial type $p\cup g\cdot p$ is defined over $A, g$ and is not algebraic.  Thus, it can be extended to a generic type over $A, g$, so there exists some $b_1=g\cdot b$ realizing $p$ generic over $A, g$ with  $b$ realizing $p$ again, see \cite[Proposition 4.1.21 \& Lemma 4.1.23]{fW00}.     
\end{enumerate}
\end{fact}

An immediate application of Fact \ref{F:G00} (c), together with Lemma \ref{L:Existsinfty}, will yield Proposition \ref{P:Sarkozy} below. For this, we need to introduce first for our purposes a definition, resonating with the combinatorial notion of \emph{intersective sets} \cite[Definition]{tL14}. 

\begin{definition}\label{D:intersective}
Consider a supersimple group $G$ of rank $1$ and a subset $W$ of $G$ definable over an elementary substructure $M$ of $G$. We say that $W$ is \emph{generically intersective} if the set $W\cap \GO M$ is infinite. 
\end{definition}

\begin{remark}\label{R:intersective}
The definition of being generically intersective does not depend on the elementary substructure over which $W$ is defined. Indeed, given another elementary substructure $M'$ over which $W$ is defined, we can find a common elementary substructure $N$ extending both $M$ and $M'$. By construction, the subgroup $\GO N$ is a subset of $\GO{M'}$, so we need only show that $W\cap \GO N$ is infinite. Otherwise, Fact \ref{F:G00} (a) and compactness yields that there is a subgroup $H$ defined over $N$ of finite index such that $W\cap H$ is finite. Since $M$ is an elementary substructure of $N$, we conclude that there is an $M$-definable subgroup $\widetilde H$ of finite index such that $W\cap \widetilde H$ is finite, contradicting that $W\cap \GO M$ is infinite. 
\end{remark}

\begin{prop}\label{P:Sarkozy}
Consider a definable generically intersective subset $W$ of a supersimple group $G$ of rank $1$. For every definable subset $X$, the set  \[ X_0=\{ x \in X \ | \ \emph{there are only finitely many $w$ in $W$ with } w\cdot x \in X \}\]
is finite.\end{prop}
Note that $X_0$ is definable by Lemma \ref{L:Existsinfty}. 
\begin{proof}
Assume otherwise for a contradiction. By Remark \ref{R:intersective}, there exists an elementary substructure $M$ over which both $W$ and $X$ are defined, and choose some $w$ in $W\cap \GO M$ generic over $M$. We can now extend the definable subset $X_0$ (seen as a partial type) to a complete generic (i.e. non-algebraic) type $p$ over $M$.  Fact \ref{F:G00} (c) applied to $p$ and $w$ yields that there exists some $x$ realizing $p$ generic over $M, w$ such that $w\cdot x$ realizes $p$ again. Thus, both $x$ and $w\cdot x$ belong to $X$. Furthermore, 
\begin{multline*}
   \dim(w/M, x) = \dim(x, w/M) - \dim(x/M) = \\ =\dim(w/M)+ \dim(x/M,w) - \dim(x/M) = 1 + 1- 1 =1. 
\end{multline*}
Hence, the type $\tp(w/M, x)$ is not algebraic so it has infinitely many realizations. Every such realization $w'$ yields that $w'\cdot x$ belongs to $X$, contradicting that $x$ belongs to $X_0$. 
\end{proof}

If $M$ is an elementary substructure of $G$, each coset of the subgroup $\GO M$  is type-definable over $M$ and 
hence $M$-invariant (it need not be the case that this coset has a representative in $M$). Thus, every type $p$ over 
$M$ contained in $G$ must determine a coset of $\GO M$. We 
denote by $\Cos p$ the coset of $\GO M$ containing some (and hence every) realization of $p$. Given three types $p$, $q$ and $r$ over $M$ such that $a\cdot b$ realizes $r$ for some realizations $a$ of $p$ and $b$ of $q$, then $\Cos p\cdot \Cos q=\Cos r$ in the quotient group $G/\GO M$. The following result can be seen as a sort of converse. 

\begin{fact}\label{F:pqr}\textup{(}\cite[Proposition 2.2]{PSW98} \& \cite[Lemma 2.3]{MPP04}\textup{)}~
Consider an elementary substructure $M$ of $G$ and three generic types $p$, $q$ and $r$ of $G$ over $M$ such that 
\[
\Cos p \cdot \Cos q=\Cos r
\]
in the quotient group $G/\GO M$. 
Then, there are realizations $a$ of $p$ and $b$ of $q$ with  $a\ind_M b$ and $a\cdot b$ realizing $r$.
\end{fact}

\begin{definition}\label{D:gen_squares}\textup{(}c.f.~\cite[Theorem 1.5]{tS17}\textup{)}~
A subset $X$ of $G$ definable  over the parameter subset $A$ has \emph{generically distinct squares} if $x^2\neq y^2$ for every two generic independent elements $x$ and $y$ of $X$ over $A$.
\end{definition}
\begin{remark}\label{R:gen_squares}~
\begin{enumerate}[(a)]
    \item If the definable set $X$ is finite, then it clearly has generically distinct squares. Similarly, if $G$ is abelian and has only finitely many involutions, every definable subset has generically distinct squares. 
\item Assume now that the $A$-definable subset $X$ of $G$ has generically distinct squares. An element $x$ in $X$ is generic over $A$ if and only if so is $x^2$, for the type $\tp(x/A, x^2)$ is algebraic in the supersimple group $G$ of rank $1$. 
\end{enumerate}
\end{remark}
 
\begin{theorem}\label{T:Main}
Consider a supersimple group $G$ of rank $1$ and a definable subset $X$ of $G$. If $X$ has generically distinct squares, then for every definable subgroup $H$ of $G$ of finite index the subset of \emph{bad initial points}  
\[ \mathrm{B}(X)=\left\{x \in X\ | \ \exists^{<\infty} g \text{ in $H$ with $x$, $x\cdot g$ and $g\cdot x\cdot g$ in $X$}\right\} 
\] is finite. 
 \end{theorem}
 As before, note that $\mathrm{B}(X)$ is definable by Lemma \ref{L:Existsinfty}. 
\begin{proof}
Consider an elementary substructure $M$ of $G$ over which both $X$ and $G$ are defined. If the $M$-definable subset $\mathrm{B}(X)$ were infinite, we can proceed as in the proof of Proposition \ref{P:Sarkozy} and complete it to a generic (that is, a non-algebraic) type $p$ over the elementary substructure $M$. Since $X$ has generically distinct squares, the type $r=\tp(a^2/M)$, with $a$  realizing $p$ (and thus in $X$), is again generic by Remark \ref{R:gen_squares} (b). 

Fact \ref{F:pqr} applied to the triple $p$, $p$ and $r$ yields two 
$M$-independent  realizations $a$ and $c$ of $p$ with $a\cdot c$ realizing $r$, so $a\cdot c=b^2$ for some $b$ realizing $p$.
Set $g= a\inv \cdot b$. Notice that the elements $a$, $b=a\cdot g$ and 
$c=g\cdot a\cdot g$ all lie in $\mathrm{B}(X)\subset X$. We need only show that $g$ belongs to $H$ and that $\tp(g/M, a)$ is not algebraic to contradict that $a$ belongs to $\mathrm{B}(X)$. Indeed, whenever $g_1$ realizes $\tp(g/M,a)$, we have that $a, a\cdot g_1$ and $g_1\cdot a\cdot g_1$ all lie in 
$X$. 

Now, the element $b$ cannot be algebraic over $M,a$, since the generic element $c=a\inv\cdot b^2$ is independent from $a$ over $M$. It follows that $g= a\inv \cdot b$ cannot be algebraic over $M,a$. Now, the elements $a$ and $b$ are two independent realizations of the same  generic type $p$, so they belong to the same coset $\Cos p$ of $\GO M$. Hence, their difference $g$ belongs to $\GO M$, which is a subgroup of $H$ (\cf Definition \ref{D:G00}), as desired. 
\end{proof}

\begin{definition}\label{D:corner}
Given an abelian group $G$ (written additively) and subset $S\subset G\times G$,  a \emph{(non-trivial) corner} for $S$ consists of a pair $(x, y)$ in $G\times G$ as well as some $g\ne 0_G$ in $G$ such that  all three pairs $(x, y)$,  $(x, y+g)$ and $(x+g, y)$ belong to $S$. 

A \emph{(non-trivial) skew-corner} for $S$ consists of a triple $(x, y, y')$ in $G^3$ and some $g\ne 0_G$ in $G$  such that the pairs $(x,y)$, $(x, y+g)$  and $(x+g, y')$ all belong to $S$. 
\end{definition}

If $S$ is a definable subset of $G\times G$ over the parameter set $A$, given an element $g$, the collection $\mathrm{sk}_{\llcorner}(S)_g$ of triples $(x, y, y')$ which build a skew-corner for $S$ with respect to $g$ is a subset of $G^3$ definable over $A, g$. In particular, we can consider the $g$'s for which $\dim(\textrm{sk}_{\llcorner}(S)_g)<3$. We will prove the desired result in this context in a slightly more general situation. 

\begin{definition}\label{D:preserves_gen}
A definable map $f:G\to G$ of the supersimple group $G$ of rank $1$ \emph{preserves generically the connected component} if, for every subset of parameters $A$ over which $f$ is defined and every generic element $g$ of $\GO A$ over $A$, we have that $f(g)$ is again a generic element of $\GO A$ over $A$.   
\end{definition}

\begin{remark}
In a supersimple abelian group of rank $1$ with finite $k$-torsion for some $k\ge 2$, the map $x\mapsto k\cdot x$ preserves generically the connected component. 
\end{remark}

\begin{prop}\label{P:skew}
Consider a supersimple abelian group $G$ of rank $1$ such that the connected component exists (\cf Fact \ref{F:G00} (b)) and a definable subset $S$ of $G^2$ of dimension $2$. Let $f_1$ and $f_2$ be two definable functions which preserve  generically the connected component. The  set
\[ 
\left\{ g\in \GO  \ \ | \ \dim(\mathcal X_g)<3 \right\}
\] is finite, where $\mathcal X_g=\mathcal X_g(S,f_1,f_2)$ denotes the set \[ \{(x, y, y') \in G^3 \ | \ (x, y), (x, y'), (x+f_1(g), y) \text{ and } (x, y'+f_2(g)) \text{ all lie in } S\}.\] 
In particular, with $f_1=f_2=\mathrm{Id}_G$, the set 
\[
\left\{ g\in \GO  \ \ | \ \dim(\mathrm{sk}_{\llcorner}(S)_g)<3 \right\}
\]
is finite.
\end{prop}
\begin{proof}
Fix an elementary substructure $M$ containing the parameters needed to define the set $S$ as well as the maps $f_1$ and $f_2$. Assume for a contradiction that the set 
\[ 
\left\{ g\in \GO  \ \ | \ \dim(\mathcal X_g)<3 \right\}
\] 
is infinite and choose $g$ in $G^{00}$ generic over $M$ such that $\mathcal X(S, f_1, f_2)_g$ has dimension strictly less than $3$. Since the definable subset $S$ has dimension $\dim(S)=2$, we can choose $(x_1,y)$ in $S$  of dimension $2$ over $M$ with $x_1, y\ind_M g$, so  $\dim(x_1,y/M,g)=2$. 
Symmetry of forking yields that the element $g$ is generic (that is, not algebraic) over $M, x_1, y$. 

By assumption $G^{00}=\GO{M,y}$, so Fact \ref{F:G00} (c) applied to the non-algebraic type $\tp(x_1/M,y)$ and the generic element $f_1(g)$ of $\GO{M,y}$ over $M, y$ yields an element $x$ in $G$ generic over $M, y, g$ such that both $x$ and $x+f_1(g)$ realize $\tp(x_1/M,y)$. Since $S$ is $M$-definable, we deduce by invariance that both pairs $(x,y)$ and $(x+f_1(g),y)$ belong to $S$. Furthermore, we have that 
\[
\dim(x,y/M, g)=\dim(x/M,g,y)+\dim(y/M,g)=1+1=2. 
\]
Our assumption on $f_2$ yields that the element $f_2(g)$ in $G^{00}=\GO{M,x}$ is generic over $M, x$ (by symmetry of non-forking).  Applying now Fact \ref{F:G00} (c) to the non-algebraic type $\tp(y/M, x)$, we can find some element $y'$ generic over $M, x, g$ such that both $y'$ and $y'+f_2(g)$ realize $\tp(y/M,x)$. We may assume that $y'\ind_{M, x, g} y$, so $\dim(y'/M, x, y, g)=1$. 

As before, we deduce that the pairs  $(x,y')$ and $(x,y'+f_2(g))$ belong to $S$, so the triple $(x, y, y')$ lies in  $\mathcal X(S, f_1, f_2)_g$.  By construction, 
\[ \dim(x, y, y'/M, g)= \dim(x,y/M, g) + \dim(y'/M, g, x, y)=2+1=3,\]
 which yields the desired contradiction to our choice of $g$. 
\end{proof}

% The above Proposition yields immediately a result on the \emph{additive energy} of definable subsets $X$ of $G$. We would like to thank Julia Wolf, who pointed out this connection in private discussion.  
% \begin{cor}\label{C:energy}
% tLet $G$ be a supersimple abelian group of rank $1$ such that its connected component exists and $X$ an infinite definable subset of $G$ of dimension $1$. The set 
% \[ 
% \left\{g\in \GO{} \ | \ \dim(E_g(X,X))<2 \right\}
% \] 
% is finite, where  
% \[
% E_g(X,X)=\left\{ (x_1,x_2,x_3,x_4) \in X^4 \ | \ x_2-x_1=g=x_3-x_4 \right\}.
% \]
% In particular, the \emph{additive energy} of $X$
% \[
% E(X,X)=\left\{ (x_1,x_2,x_3,x_4) \in X^4 \ | \ x_1+x_3=x_2+x_4 \right\}
% \] has dimension $3$. 
% \end{cor}
% \begin{proof}
% Given an infinite definable subset $X$ of $G$, its \emph{Cayley graph} \[ S=\{(x,y) \in G\times G \ | \  x-y \in X\}\]  has dimension $2$. Assume for a contradiction that there is a generic element $g$ in $\GO{}$ such that $E_g(X, X)$ has dimension strictly less than $2$.   Proposition \ref{P:skew} yields that $\mathcal{X}(S)_g$ has dimension $3$, so choose a triple $(x, y, y')$ of maximal dimension over $M, g$. The tuple consisting of the elements 
% \begin{align*}
%     x_1&=x-y &  x_3&=x-y' \\  x_2& =x+g-y  &  x_4& =x-y'-g. 
% \end{align*}
% lies in $E_g(X, X)$ and has dimension $2$ over $M, g$, which gives the desired contradiction.
% \end{proof}

\section{Some additive configurations in expansions of the integers}\label{S:integers}

Kaplan and Shelah \cite{KS17} showed that the theory of the additive group $\Z$ together with a unary predicate $\PR$ for the prime integers is supersimple of rank $1$, assuming Dickson's conjecture. A similar result was shown unconditionally by Bhardwaj and Tran \cite{BT21} for the theory of the additive group $\Z$ together with a unary predicate $\SF$ for the square-free integers. One of the key aspects of both proofs was a partial quantifier-elimination which allows to reduce the study of definable sets to those of a certain shape. In this section, we will present the main results from \cite{KS17, BT21} (see also \cite{MPP25} for a self-contained exposition of their results) in a unified manner and deduce some consequences of combinatorial nature using the general results of Section \ref{S:SS}. 

In the following, the predicate $P$ denotes either the subset $\PR$ of prime integers (assuming Dickson's conjecture) or the subset $\SF$ of square-free integers. Analogously, given a natural number $n\ge 1$, set 
\[ P_n =\{ k\in n\Z \ | \ k/n \text{ belongs to } P\}  .
\]

\begin{definition}\label{D:basic_subset}
A \emph{constructible} subset of $\Z^{|\bar x|+|\bar y|}$ in the partition $(\bar x; \bar y)$  is a finite union of \emph{basic} subsets $X(\bar x, \bar y)$, where each basic subset is given by some finite subsets of indexes $i$ and $k$, some $\epsilon_{i,k}$'s in $\{-1,1\}$, a finite collection of affine-linear forms \[ t_{i, k}(\bar x, \bar y)=c_{i,k}+\sum_{j=1}^{|\bar x|} m_{i,k,j} x_j + \sum_{\ell=1}^{|\bar y|} m'_{i, k, \ell} y_\ell\ \] with $c_{i,k}$ an integer, both $m_{i, k,j}$ and $m'_{i,k, j}$ in $\Z$ as well as a finite system of congruences $\Sigma(\bar x, \bar y)$ involving some of the above forms such that
\[(\bar x, \bar y) \in X \ \iff \  \Sigma(\bar x, \bar y) \text{ holds with } \bigwedge\limits_{i,k}  t_{i,k}(\bar x, \bar y) \in P_k^{\epsilon_{i, k}} ,\]  with the convention that $P_k^{-1}=\Z\setminus P_k$. 

In an abuse of notation, given a finite tuple $\bar b$ of length $|\bar y|$ in the subset $B$ of parameters, the corresponding subsets of $\Z^{|\bar x|}$ given by $X(\bar x, \bar b)$ are \emph{basic}, resp. \emph{constructible}, subsets defined over $B$. 
\end{definition}
Note that the family of all constructible subsets of $\Z^n$ defined over a subset $B$ is closed under Boolean combinations. 

The following fact contains the essence of the results of \cite{KS17, BT21} for the purposes of this section.
\begin{fact}\textup{(}\cf \cite[Proposition 3.1 \& Theorem 3.3]{MPP25}\textup{)}~\label{F:basics_qe}
\begin{enumerate}[(a)]
    \item In each of the structures $(\Z,0, 1, +, -,\PR)$ (assuming Dickson's conjecture) and $(\Z,0, 1, +, -,\SF)$, the collection of constructible subsets coincide with the parameter-free definable subsets. Hence, the projection of a constructible subset is again constructible.  
    \item The theory of each of the structures $(\Z,0, 1, +, -,\PR)$ (assuming Dickson's conjecture) and $(\Z,0, 1, +, -,\SF)$ is supersimple of rank $1$. In particular, by Lemma \ref{L:Existsinfty}, for every constructible subset $X(x; \bar y)$  there exists some $M_X$ in $\N$ such that for every tuple $\bar b$ in $\Z^{|\bar y|}$, either $X(x, \bar b)$ has size at most $M_X$ or it is infinite. 
    \item Within any given model of each of the  theories of the structures $(\Z,0, 1, +, \PR)$ (assuming Dickson's conjecture) and $(\Z,0, 1, +, \SF)$,  the model-theoretic algebraic closure of a subset $B$ of parameters coincide with the pure closure $\strp {B,1}$ of the subgroup $\sbgp{B, 1}$ generated by $B$ and $1$. 
\end{enumerate}
\end{fact}

\begin{notation}
For every $C>0$, there are only finitely many possible basic subsets which can be defined by a formula of \emph{complexity} at most $C$, where the complexity of a formula is computed in terms of its length  in the language $\{0, 1, +, -, P\}$.  
\end{notation}

A straight-forward application of Proposition \ref{P:Sarkozy} (using the above fact) yields the following results: 

\begin{cor}\label{C:Sarkozy_primes}\textup{(}\cf \cite{aS78}\textup{)}~
Assume Dickson's conjecture. For any complexity $C>0$, there are constants $M$ and $N$ in $\N$ such that for any constructible subset $X(x;\bar y)$ in the structure $(\Z, 0, +, \PR)$ of complexity at most $C$ and every choice $\bar b$ of parameters, 
the set \[ \{ a \in X(x, \bar b)\ | \ \ \exists^{\le M} \text{ $p$ in $\PR$ with  $a + (p-1)$ in $X(x, \bar b)$} \}\]
is finite of size at most $N$. 

In particular, if the constructible subset $X(x, \bar b)$ is infinite, then for  all but finitely many $a$'s in $X(x, \bar b)$, there are infinitely many primes $p$ with $a+(p-1)$ in $X(x, \bar b)$. 
\end{cor}
\begin{proof}
By Fact \ref{F:basics_qe}, for every constructible $X(x, \bar y)$ of complexity at most $C>0$, the subset $Z(u; x, \bar y)= X(x+u;\bar y)$ is again definable and thus constructible of complexity $D=O_C(1)$. With the notation of Fact \ref{F:basics_qe} (b),  choose $M=\max(M_Z)$ where $Z$ runs over all possible constructible subsets of complexity at most $D$. Again using Fact \ref{F:basics_qe}, the subset \[ W(x; \bar y)= X(x, \bar y) \cap \exists^{\le M} u \  Z(u, x, \bar y)\] is constructible of complexity $D_1=O_C(1)$, so choose  $N=\max(N_W)$, where $W$ runs over all possible constructible subsets of bounded complexity as above.

By Dirichlet's theorem on arithmetic progressions, given $d\ge 2$, infinitely many primes can be written as $1+dn$ for some $n$ in $\N$, so the definable set $P-1$ intersects the subgroup $d\Z$ in an infinite set. Passing to an elementary extension $G$ of the structure $(\Z, 0, +, P)$, the connected component $\GO A$ does not depend on the parameter set $A$ by Fact \ref{F:G00} (b) and equals $\bigcap_{d\ge 2} dG$. An easy compactness argument yields that the corresponding definable set $P-1$ in $G$ is generically intersective, so we conclude the result by Proposition \ref{P:Sarkozy} and our choice of $M$ and $N$.
\end{proof}

Primes are square-free integers, so Dirichlet's theorem and the proof of Corollary \ref{C:Sarkozy_primes} yields the following result, without assuming Dickson's conjecture. 

\begin{cor}\label{C:Sarkozy_sq}
For any complexity $C>0$, there are constants $M$ and $N$ in $\N$ such that for any constructible subset $X(x;\bar y)$ in the structure $(\Z, 0, +, \SF)$ of complexity at most $C$ and every choice $\bar b$ of parameters, 
the set \[ \{ a \in X(x, \bar b)\ | \ \exists^{\le M} n  \text{  in $\SF$ with  $a + (n-1)$ in $X(x, \bar b)$} \}\]
is finite of size at most $N$. 

In particular, if the constructible subset $X(x, \bar b)$ is infinite, then for  all but finitely many $a$'s in $X(x, \bar b)$, there are infinitely many square-free $n$ with $a+(n-1)$ in $X(x, \bar b)$.  \qed
\end{cor}

As a consequence of Remark \ref{R:gen_squares} (a), Theorem \ref{T:Main} applies to every constructible subset $X(x, \bar y)$ in the supersimple abelian group $(\Z, 0, +, P)$. In particular, we obtain a stronger version of van der Corput's theorem on 3-AP's, by  Example \ref{E:s1_theories}:  

\begin{cor}\label{C:Main}
Assume Dickson's conjecture. For any complexity $C>0$ and every integer $k\ge 1$, there are constants $M$ and $N$ in $\N$ such that for every constructible subset $X(x;\bar y)$ in the structure $(\Z, 0, +, \PR)$ of complexity at most $C$ and every choice $\bar b$ of parameters, the set 

\[ \left\{a \in X(x, \bar b) \ | \ \exists^{\le M} g  \text{ in  $k\Z$  with $a$, $a+ g$ and $a+2g$ in $X(x, \bar b)$}\right\} 
\]
is finite of size at most $N$. In particular, if the constructible subset $X(x, \bar b)$ is infinite, all but finitely elements in $X$ are the initial points of infinitely many arithmetic progressions in $X$ of length $3$ with distance divisible by $k$.
\end{cor}

The same statement holds unconditionally in the structure $(\Z, 0, +, \SF)$, without assuming Dickson's conjecture. 

\begin{remark}\label{R:acl_hyperplane}
Fix some ambient model $\UU$ of either the theory $\mathrm{Th}(\Z, 0, +, \PR)$ (assuming Dickson's conjecture) or  $\mathrm{Th}(\Z, 0, +, \SF)$ and a subset $B$ of parameters.  Fact \ref{F:basics_qe} (c) and the definition of the dimension in a supersimple theory of rank $1$ in terms of the matroid given by algebraic closure yield that a $n$-tuple $\bar a$ of $\UU$ does not have full dimension over  $B$  if and only if $\bar a$ belongs to some (non-trivial) \emph{affine hyperplane} defined over $B$, that is, 
\[ \bar a \in \mathcal H_{t(\bar b)}=\bigg\{ \bar x \in \UU^{n} \ | \ \sum_{j=1}^{n} m_j x_j= t(\bar b)\bigg\},\] for integers $m_i$ (not all zero), some affine-linear form $t(\bar y)$ and a finite tuple $\bar b$ in $B$.  

Given $k$ many  affine-linear forms $t_i(\bar y_i)$ and tuples of integers $\bar m_i$, the family of finite unions 
\[ \bigcup_{i=1}^k \mathcal H_{t_i(\bar c_i)}\]
of their associated affine hyperplanes, where each $\bar c_i$ runs over $\UU^{|\bar y_i|}$, is uniformly definable. 
\end{remark}
A straight-forward compactness argument together with Remarks \ref{R:dim} (c) and \ref{R:acl_hyperplane} yields the following result. 
\begin{lemma}\label{L:hyperplanes_dim}
Given a constructible subset $X(\bar x,\bar y)$ of $\Z^{|\bar x|+|\bar y|}$, there are finitely many affine hyperplanes $\mathcal H_{t_i(\bar y)}$ such that for every tuple $\bar b$ in $\Z^{|\bar y|}$ the subset $X(\bar x,\bar b)$ of  $\Z^{|\bar x|}$ has full dimension if and only if \[ X(\bar x,\bar b) \not\subset  \bigcup_{i=1}^k \mathcal H_{t_i(\bar b)}.\]
In particular, if the latter conditions holds for $\bar b$, then $X(\bar x,\bar b)$ is never contained in any finite union of affine hyperplanes defined over $\bar b$.
\qed
\end{lemma}
We will now conclude this section with the following result, using  Fact \ref{F:G00} (b), Proposition \ref{P:skew} and the above lemma. 

\begin{cor}\label{C:skew_Z}
Given  a constructible subset $S(z_1,z_2, \bar u)$ in the structure $(\Z,0,+,\PR)$ (assuming Dickson's conjecture), there are:
\begin{itemize}
    \item a natural numbers $k\ge 1$,
    \item finitely many affine hyperplanes $\mathcal H_{t_i(\bar u)}$ of $\Z^2$ and
    \item finitely many affine hyperplanes $\mathcal H_{r_j(\bar u)}$ of $\Z^3$
\end{itemize}
such that for every tuple $\bar b$, whenever $S(z_1, z_2,\bar b)$ is not contained in $\bigcup_i \mathcal H_{t_i(\bar b)}$, then the subset
\[ 
\left\{ g\in k\Z  \ \ | \ \mathcal X_g \subset \bigcup\nolimits_j \mathcal H_{r_j(\bar b)} \right\}
\] is finite, where \[ \mathcal X_g=\{(x, y, y') \in \Z^3 \ | \ (x, y), (x, y'), (x+g, y) \text{ and } (x, y'+g) \text{ all lie in } S(\bar z, \bar b)\} .\]
% In particular, if the constructible subset $S(z_1,z_2, \bar b)$ is not contained in $\bigcup_i \mathcal H_{t_i(\bar b)}$, then for all but finitely many $g$'s in $k\Z$ and  %affine hyperplanes $\mathcal H_{s(\bar b)}$, 
% finitely many affine-linear forms $s_j(\bar u)$, the set  $\mathcal X_g$ is not contained in the finite union of affine hyperplanes $\mathcal H_{s_j(\bar b)}$.
% %defined over $\bar b$.
\end{cor}
There is an analogous unconditional result for constructible subsets in the structure  $(\Z,0,+,\SF)$, without assuming Dickson's conjecture. 
\begin{proof}
 For the constructible subset $S(z_1, z_2, \bar u)$, consider the corresponding constructible subset \[ 
 X(x, y, y'; v, \bar u)= (x, y) \in S \land  (x, y') \in S \land (x+v, y) \in S\land  (x, y'+v) \in S.
 \] 
 Choose finitely many affine hyperplanes $\mathcal H_{t_i(\bar u)}$ of $\Z^2$ for $S$ and $\mathcal H_{r_j(\bar u)}$ of $\Z^3$ for $X$ as in Lemma \ref{L:hyperplanes_dim}. For every $k$ in $\N$ the constructible set 
 \[ 
 Y_k(\bar u)= \exists v_1 \cdots \exists v_k \in k!\Z \text{  with $v_\ell\neq v_{\ell'}$ and each }  X(x, y, y', v_\ell, \bar u)\subset \bigcup\nolimits_j \mathcal H_{r_j(\bar u)}
 \]
witnesses that there are at least $k$ many $g$'s in $k!\Z$ with $X(x, y, y', g, \bar u)$ of dimension strictly less than $3$. Similarly, let $W(\bar u)$ be the constructible subset expressing that $S(\bar z, \bar u)$ is not contained in $\bigcup\nolimits_i \mathcal H_{t_i(\bar u)} $. 
If no  member in the decreasing chain 
\[   
W \cap Y_{1} \supset \cdots \supset  W\cap Y_{k} \supset \cdots  
\] 
is the empty set, then a standard compactness argument produces a tuple $\bar b$  in some elementary extension $G$ of the structure $(\Z,0,+,\PR)$ and infinitely many $g$'s in \[ \bigcap_k k!G \stackrel{\text{\ref{F:G00} (b)}}{=} \GO{}\] with $\dim(S(x_1, x_2, \bar b))=2$ yet $\dim(\mathcal X_g)<3$, contradicting Proposition \ref{P:skew}. Hence for some $k\ge 1$ in $\N$, the set $W$ is contained in the complement of $Y_k$, which yields the desired result (replacing $k$ with $k!$ in the statement).  

The last assertion in the statement follows directly from Lemma \ref{L:hyperplanes_dim}. 
\end{proof}

\section{Some additive configurations in (pseudo-)finite fields}\label{S:psf}

A \emph{pseudo-finite field} is an infinite model of the common theory of all finite fields. Ax's work on the elementary theory of finite fields \cite[Section 8]{jA68} yields that the axioms expressing being pseudo-finite are elementary and the completions are given by the isomorphism type of the relative algebraic closure of the prime field. 

In the seminal work of Chatzidakis, van den Dries and Macintyre \cite{CDM92}, the asymptotic theory of finite fields was studied from a model-theoretic point of view within Shelah's classification spirit. In particular, they showed that every pseudofinite field in the language of rings has a supersimple theory of rank $1$ \cite[Proposition 4.5]{CDM92}. 
Moreover, they used Lang-Weil estimates together with a partial quantifier elimination to conclude their striking result (see Fact \ref{F:CDM} below). 

Whilst there is a partial quantifier elimination valid for all finite fields \cite[Lemma 2.9]{CDM92}, we nevertheless believe it will be more convenient to work with the notion of definable sets, instead of formulating the results for \emph{basic/constructible} sets as in Section \ref{S:integers}. Similar to the use of complexity in the previous section, we will refer to the complexity of the definable sets in terms of the length of the corresponding $\LL_{\rm ring}$-formula as a sequence of symbols (\cf \cite[p. 30]{tT15}). Note that there are only finitely many formulae of a given complexity.

\begin{fact}\label{F:CDM}\textup{(}\cite[Main Theorem 3.7 \& Proposition 3.8]{CDM92}\textup{)}
For every complexity $C>0$, there are finitely many values $(d, \mu)$ with $d$ in $\N$ and $\mu$ in $\Q^{>0}$ such that the following holds. 
\begin{enumerate}[(a)]
    \item For every $\LL_{\rm ring}$-formula $\varphi(x_1,\ldots, x_n;y_1,\ldots, y_m)$ of complexity at most $C$ and every $m$-tuple $b$ in the finite field $k$, whenever the set $\varphi(k^n, b)$ is not empty, then there is a pair $(d, \mu)$ with $0\le d\le n$ and 
\[ 
\big|\varphi(k^n, b) \big|= \mu |k|^d   + O_C\big(|k|^{d-\frac{1}{2}}\big). 
\]  
\item For every such pair $(d, \mu)$ there is a formula $\psi(y_1,\ldots, y_m)$ of complexity $O_C(1)$ such that for every finite field $k$, an $m$-tuple $b$ satisfies $\psi$ if and only if the (cardinality of the) non-empty set $\varphi(k^n, b)$ satisfies the above equation. 
\item  For an $\LL_{\rm ring}$-formula $\varphi(x;y_1,\ldots, y_m)$ of complexity at most $C$, there are constants $\ell_C$ in $\N$ and $0<\delta_C\le 1$  such that \[ \text{either }  \big|\varphi(k, b) \big|\le \ell_C \text{ or }  \big|\varphi(k, b) \big|\ge \delta_C |k|\]  for every $m$-tuple $b$ in the finite field $k$.
\end{enumerate} 
\end{fact}
The existence of the pairs of constants $(d, \mu)$ follows from the study of the corresponding definable set in the ultraproduct of finite fields. In fact, given a definable set $X\subset K^n$, where $K$ is a pseudo-finite field, the value $d$ equals the dimension $\dim(X)$ of $X$, as in Remark \ref{R:dim}, and $\mu$ equals its \emph{density} when $n=1=\dim(X)$ (see \cite[p.130]{CDM92} for further details). 

Running over all pairs $(d, \mu)$ with $d<n$ in Fact \ref{F:CDM} (b), one deduces the following result:
\begin{cor}\label{C:CDM_maxdim}
For an $\LL_{\rm ring}$-formula $\varphi(x_1,\ldots, x_n;y_1,\ldots, y_m)$ of complexity at most $C$, there exists a formula $\rho(y_1,\ldots, y_m)$ of complexity $O_C(1)$ and some value $0<\delta_C\le 1$ such that for every finite field $k$, 

\begin{align*}
    \rho(k) & =\left\{ b \in k^m \ | \ |\varphi(k^n, b)| =\mu |k|^d+ O_C(|k|^{d-\frac{1}{2}})  \text{ for some $0\le d< n$ and $\mu>0$} \right\}   \\& =\left\{ b \in k^m \ | \ |\varphi(k^n, b)|< \delta_C|k|^n \right\}     .
\end{align*}
% \[ 
% \rho(k)=\left\{ b \in k^m \ | \ |\varphi(k^n, b)|=o(|k|^n) \right\} .       
% \]
% an $m$-tuple $b$ satisfies $\rho$ if and only if the dimension $d$ of the non-empty set $\varphi(k^n, b)$ is strictly less than $n$, that is $|\varphi(k^n, b)|=o(k^n)$. 
\end{cor}

\subsection*{Roth's theorem on $3$-AP's for (pseudo-)finite fields}\label{S:3AP_psf}

Using a standard ultraproduct construction (often referred to as \emph{pseudo-finite yoga}), we obtain a finitary version of Theorem \ref{T:Main}. 

\begin{theorem}\label{T:Roth_fte}
Given a complexity $C>0$ and natural numbers $r$ and $s$, there is a constant $t=t(C, r, s)$ in $\N$ such that for every finite field $k$, every definable additive subgroup $H$ of $k$ of index at most $s$ and every definable subset $X$ of $k$, whenever both $X$ and $H$ have complexity at most $C$, then the set of \emph{bad} initial points
\[ 
\mathrm{B}_{r}(X)=\left\{x\in X\ | \ \exists^{\le r} g \text{ in $H$ with $x, x+g$ and $x+2g$ in $X$} \right\}
\] 
has size at most $t$.  Moreover, there are two  constants $\eta=\eta(C,s)>0$ in $\N$ and $\delta=\delta(C)>0$ such that whenever $|X|\ge \eta$, the subset  
\[ 
X_1=\left\{x\in X\ | \ \exists^{\ge  \delta|k|} g  \text{ in $H$ with $x, x+g$ and $x+2g$ in $X$} \right\}
\] has cardinality at least $\delta |k|$. 
\end{theorem}

\begin{proof}
 For every two formulae  $\varphi(x,y)$ and $\vartheta(z,y)$
both of complexity at most $C$, the formula 
\[ 
\psi(z,x,y)= \varphi(x, y)\land \varphi(x+z, y)\land \varphi(x+2z, y) \land \vartheta(z, y) 
\] 
has complexity at most $D=O_C(1)$. Fact \ref{F:CDM} (c) yields constants $\ell_D$ and $\delta_D$ (only depending on $C$) such that for fixed $x$ and every $|y|$-tuple $b$ of a finite field $k$, the cardinality of the set $\psi(k, x, b)$ is either bounded by $\ell_D$ or is at least $\delta_D\cdot|k|$. Now, the formulae \[ \theta_r(x, y)= \varphi(x, y) \land \exists^{\le r} z \psi(z, x, y) \ \text{ and } \ \chi_r=\varphi\land\neg \theta_r \] have both complexity at most $D_1=D_1(C, r)$, so Fact \ref{F:CDM} (c) again yields the corresponding values $\ell_{D_1}$ and $\delta_{D_1}$ for $\theta_r$ and $\chi_r$.

In order to show the first assertion in the theorem, it suffices to show that for every fixed complexity $C>0$ and every natural numbers $r$ and $s$, there exists a value $n_0$ in $\N$ such that for every finite field $k$ of size at least $n_0$, the statement holds with $t=\ell_{D_1}$. Indeed, it then suffices to set $t'$ the maximum of $t$ and $n_0$.  

Assume now otherwise, so there are natural numbers $r$ and $s$ and some complexity $C>0$ such for every natural number $n$ there exists a finite field $k_n$ of size $|k_n|\ge n$ as well as an $\LL_{\rm ring}$-definable additive subgroup $H_n$ of index at most $s$ and an $\LL_{\rm ring}$-definable subset $X_n$ of $k_n$, given respectively by the instances $\vartheta_n(z, b_n)$ and $\varphi_n(x, b_n)$ for some tuple $b_n$ in $k_n$ of two formulae of complexity at most $C$, such that the corresponding set $\mathrm{B}_{r}(X_n)$ in $k_n$ has size at least $\ell_{D_1}+1$. It follows then from Fact \ref{F:CDM} (c) that the cardinality of $\mathrm{B}_{r}(X_n)$ is at least $\delta_{D_1}|k_n|$.

Let $\mathcal U$ be a non-principal ultrafilter on $\N$. As there are only finitely many formulae of complexity at most $C$, there exists a fixed single formula $\varphi(x, y)$ such that for $\mathcal U$-almost all $n$ in $\N$, the set $X_n$ is defined by $\varphi(x,b_n)$. Similarly, we may assume that the subgroup $H_n$ is defined by the formula $\vartheta(z,b_n)$ for $\mathcal U$-almost all $n$ in $\N$.

In the pseudo-finite field $K=\prod_{n\to \mathcal U} k_n$ consider the tuple $b=[(b_n)_{n\in \N}]_{\mathcal U}$ of $K^{|y|}$. Denote now by $H$ the definable additive subgroup $\vartheta(K, b)$ and by $X$ the definable set $\varphi(K,b)$. By \L o\'s's Theorem, the definable subset $\mathrm{B}_{r}(X)$ of $K$ is infinite,  and thus so is $X$. Moreover, the subgroup $H$ of $K$  has finite index at most $s$. 

Note that the characteristic of $K$ cannot be $2$. Otherwise, since $K$ is a finite union of cosets of $H$, there exists some $x_0$ in $\mathrm{B}_{r}(X)$ with $H\cap \mathrm{B}_{r}(X)-x_0$ infinite, so  $H\cap X-x_0$ is infinite as well. Any element $g$ in the latter intersection yields that both $x_0+g$ and $x_0=x_0+2g$ belong to $X$ as well, which contradicts the choice of $x_0$ in $\mathrm{B}_{r}(X)$. 

As a consequence, the subset $X$ of the additive group $K^+$ has generically distinct squares by Remark \ref{R:gen_squares} (a). Theorem \ref{T:Main} yields that the set 
\[ 
\mathrm{B}(X)=\left\{x\in X\ | \ \exists^{<\infty} g \text{ in $H$ with $x, x+g$ and $x+2g$ in $X$} \right\}
\]
is finite. Since the infinite set $\mathrm B_r(X)$ is clearly contained in $\mathrm B(X)$, we obtain the desired contradiction. This finishes the first part of the proof.

In order to show the \textquoteleft moreover clause\textquoteright, let now $k$ be a finite field and $\varphi(x,y)$ a formula of complexity at most $C$. Given a tuple $b$ in $k$, denote by $X$ the subset $\varphi(k,b)$.  It follows from the first part of the statement that there exists some value $t$ for $r=\ell_D$ (so $t$ only depends on $C$ and $s$) such that the corresponding subset $\mathrm B_{\ell_D}(X)$ has size at most $t$. 

Consider the corresponding formulae $\theta_{\ell_D}$ and $\chi_{\ell_D}$. Notice that their complexity $D_1=D_1(C)$ only depends on the complexity $C$, and similarly we have that the value $\delta=\min(\delta_D, \delta_{D_1})>0$ only depends on $C$. Fact \ref{F:CDM} (c) applied to the formula $\theta_{\ell_D}$ of complexity at most $D_1$ yields that $|\theta_{\ell_D}(k, b)|\le \ell_{D_1}$ whenever
\[ 
|k|\ge \frac{t+1}{\delta}\ge \frac{t+1}{\delta_{D_1}}.
\] 
The set $X$ is the disjoint union of $\theta_{\ell_D}(k, b)$ and $\chi_{\ell_D}(k, b)$. Clearly, the subset $X_1$  contains $\chi_{\ell_D}(k, b)=X\setminus \mathrm{B}_{\ell_D}(X)$ (again by Fact \ref{F:CDM} (c) applied to the formula $\psi$). Thus, if $|X|\ge 2\ell_{D_1}+1$, so is $|\chi_{\ell_D}(k, b)|\ge \ell_{D_1}+1$ and thus 

\[ 
|X_1|\ge |\chi_{\ell_D}(k, b)|\ge \delta_{D_1} |k|\ge \delta|k|,
\]  as desired. Therefore, it suffices to take $\eta=\max(2\ell_{D_1}+1,\frac{t+1}{\delta})$. 
\end{proof}

\subsection*{S\'ark\"ozy's theorem for (pseudo-)finite fields}\label{S:Sarkozy_psf}

In \cite{aS78}, S\'ark\"ozy showed that a subset of $\N$ with positive upper density contains two distinct elements whose difference is a perfect square. An alternative proof was given by Furstenberg \cite{hF77} using ergodic theory. Now, the subgroup of non-zero squares in a pseudo-finite field has finite multiplicative index, so it is generically intersective \cite[Lemma 3.2]{PSW98} as in the statement of Proposition \ref{P:Sarkozy}, which allows us to obtain an asymptotic version of Proposition \ref{P:Sarkozy} for finite fields. We would like to thank Haydar G\"oral for pointing out the possibility of a connection between S\'ark\"ozy's theorem and our techniques.

\begin{theorem}\label{T:Sarkozy_subgroup_fte}
Given a complexity $C>0$, there are constants  $r_C$ and $t_C$ in $\N$ such that for every finite field $k$ of size at least $n_0$ and every $\LL_{\rm ring}$-definable subsets $X$ and $H$ of $k$ of complexity at most $C$ with $H$ a multiplicative subgroup of $k^\times$,  either $H$ has size at most $\ell_C$ (as in Fact \ref{F:CDM} (c)) or
the set
\[ 
\left\{x\in X\ | \ \exists^{\le r_C} h \text{ in $H$ with $x+h$ in $X$} \right\}
\] has size at most $t_C$.  Moreover, there is a constant $\delta=O_C(1)>0$  such that whenever  $|X|\ge 2t_C+1$ and $|H|\ge \ell_C+1$, the subset  
\[ 
X_1 = \left\{x\in X\ | \ \exists^{\ge  \delta|k|} h \text{ in $H$ with $x+h$ in $X$} \right\}
\] has cardinality at least $\delta |k|$. In particular there exists some $h$ in $H$ with $x+h$ in $X$ for all $x$ in $X_1$.  
\end{theorem}

\begin{proof}
The proof is an immediate application of Proposition \ref{P:Sarkozy} and pseudo-finite yoga as in the proof of Theorem \ref{T:Roth_fte} considering the formulae 
\[ 
\psi(z,x,y)= \varphi(x, y)\land \varphi(x+z, y)\land \xi(w, y)
\] 
as well as  
\[ 
\theta(x, y)= \varphi(x, y) \land \exists^{\le \ell_D} z \psi(z, x, y) 
\text{ and  } \chi=\varphi\land \neg \theta,
\] 
where the definable subsets $X$ and $H$ are given by instances of the formulae $\varphi(x, y)$ and $\xi(w, y)$ respectively. Indeed, if in a given finite field $k$ the subgroup $H=\xi(k, b)$ has size at least $\ell_C+1$, then Fact \ref{C:CDM_maxdim} (c) yields that the multiplicative index of $H$ in $k^\times$ is at most $\delta_C$. Thus, following the lines of the proof of Theorem \ref{T:Roth_fte}, we would have that the corresponding multiplicative definable subgroup in the ultraproduct has finite index and so it is generically intersective by  \cite[Lemma 3.2]{PSW98} (see also \cite[Lemma 2.4]{MPP04}), which gives the desired contradiction together with Proposition \ref{P:Sarkozy}.
\end{proof}

In contrast to the above, we need to impose additional conditions when the set of \emph{distances} of elements of $X$ is not contained in a multiplicative subgroup.  Whenever the characteristic in large enough, we will get a similar result for arbitrary sets. Indeed, as the additive group of a field of characteristic $0$ is divisible, Fact \ref{F:G00} (b) yields that in every pseudo-finite field $K$ of characteristic $0$, the connected component of the additive group $K^+$ exists and equals $K$ itself. In particular, every infinite definable subset of $K$ will be immediately generically intersective. This observation yields an asymptotic result at the finitary level if the characteristic of the field is sufficiently large.

\begin{theorem}\label{T:Sarkozy_fte}
Given a complexity $C>0$, there are constants $n_0$,  $r_C$ and $t_C$ in $\N$ such that for every finite field $k$ of characteristic at least $n_0$ and every $\LL_{\rm ring}$-definable subsets $X$ and $W$ of $k$ of complexity at most $C$, either $W$ has size at most $\ell_C$ (as in Fact \ref{F:CDM} (c)) or the set
\[ 
\left\{x\in X\ | \ \exists^{\le r_C} w \text{ in $W$ with $x+w$ in $X$} \right\}
\] has size at most $t_C$.  Moreover, there is a constant $\delta=O_C(1)>0$  such that whenever  $|X|\ge 2t_C+1$ and $|W|\ge \ell_C+1$, the subset  
\[ 
X_1 = \left\{x\in X\ | \ \exists^{\ge  \delta|k|} w \text{ in $W$ with $x+w$ in $X$} \right\}
\] has cardinality at least $\delta |k|$. 
\end{theorem}
In particular,  for a non-constant polynomial $f(\t)$ over the finite field $k$ of degree at most $D$, the image $W=f(k)$ is a definable subset of complexity bounded in terms of $D$. Whenever $|k|\ge D\cdot \ell_C+1$, the subset $W$ cannot have size at most $\ell_C$. Thus, the above statement yields a S\'ark\"ozy-like statement restricted to definable subsets.  
\begin{proof}
The proof follows from a straightforward adaptation of the proof above, considering again the formulae 
\[ 
\psi(z,x,y)= \varphi(x, y)\land \varphi(x+z, y)\land \xi(w, y)
\] 
as well as  
\[ 
\theta(x, y)= \varphi(x, y) \land \exists^{\le \ell_D} z \psi(z, x, y) 
\text{ and  } \chi=\varphi\land \neg \theta,
\] 
where now the definable subsets $X$ and $W$ are given by instances of the formulae $\varphi(x, y)$ and $\xi(w, y)$ respectively. 
% The proof is an immediate application of Proposition \ref{P:Sarkozy} and pseudo-finite yoga as in the proof of Theorem \ref{T:Roth_fte} considering the formulae 
% \[ 
% \psi(z,x,y)= \varphi(x, y)\land \varphi(x+z, y)\land \xi(w, y)
% \] 
% as well as  
% \[ 
% \theta(x, y)= \varphi(x, y) \land \exists^{\le \ell_D} z \psi(z, x, y) 
% \text{ and  } \chi=\varphi\land \neg \theta,
% \] 
% where the definable subsets $X$ and $W$ are given by instances of the formulae $\varphi(x, y)$ and $\xi(w, y)$ respectively. 
\end{proof}

\subsection*{Skew-corners in (pseudo-)finite fields} 
In this last subsection we prove that Proposition \ref{P:skew} and Corollary \ref{C:CDM_maxdim} allow us to extract a finitary version for pseudo-finite fields using pseudo-finite yoga, for which we will include again all the details for the sake of the presentation.

As noted in the previous subsection on S\'ark\"ozy's theorem, in every pseudo-finite field $K$ of characteristic $0$, the connected component of the additive group $K^+$ exists and equals $K$ itself. In particular, every non-constant polynomial $P(\t)$ preserves generically the connected component, since the element $g$ and its image $P(g)$ are interalgebraic over the coefficients of $P$.  
\begin{theorem}\label{T:Skew_fte}
Given a complexity $C>0$, there are constants $n_0$ and $t_C$ in $\N$ as well as  $\delta_C>0$ and $\delta'_C>0$ in $\Q$ such that for every finite field $k$ of characteristic at least $n_0$, every $\LL_{\rm ring}$-definable subset $S$ of $k^2$ of complexity at most $C$ and non-constant polynomials $P_1$ and $P_2$ of degree at most $C$ with coefficients in $k$, 
\[
\text{either } |S|<\delta_C |k|^2 \quad\text{or the set}\quad 
\left|\left\{g \in k \ | \  |\mathcal X_g|<\delta'_C|k|^3  \right\}\right| \le t_C,
\] 
where \[ \mathcal X_g =\{(x, y, y') \in k^3 \ | \ (x, y), (x, y'), (x+P_1(g), y) \text{ and } (x, y'+P_2(g)) \text{ all lie in } S\}. \] 
\end{theorem}

\begin{proof}
Notice that a polynomial $P(\t)=\sum_{i=0}^C a_i \t^i$ of degree at most $C$ can be identified with its tuple $\bar a$ of coefficients. In particular, the collection of all non-constant polynomials of degree at most $C$ is an $\LL_\mathrm{ring}$-definable subset of $k^{C+1}$.  We will use in some formulae the compact notation $\bar a(g)$ for $P(g)=\sum\nolimits_{i=0}^C a_i g^i$.

Assume that an arbitrary $S$ is defined by an instance of the formula $\varphi(x,y,\bar z)$, of complexity $C$. The formula 
\[ 
\chi(x,y,y',v, \bar z,\bar u_1,\bar u_2)= \varphi(x, y,\bar z)\land \varphi(x, y',\bar z)\land \varphi(x+ \bar u_1(v), y,\bar z)\land\varphi(x, y'+\bar u_2(v),\bar z)
\] 
 has complexity $D=O_C(1)$ and describes the corresponding set $\mathcal X_v$ in each finite field. By Corollary \ref{C:CDM_maxdim}, there exists an $\LL_\mathrm{ring}$-formula $\rho'( v, \bar z,\bar u_1, \bar u_2)$ and a constant $\delta'_C$ associated to $\mathcal X_v$ such that the realizations of $\rho'$ consist of the parameters $(g, \bar c, P_1, P_2)$ for which the cardinality of the corresponding set $\mathcal X_g$ does not have full dimension. Furthermore, the complexity of $\rho'$ is $D_1=O_C(1)$.

Similarly, let $\rho(\bar z)$ and $\delta_C$ be the $\LL_\mathrm{ring}$-formula and the constant associated to $\varphi(x, y, \bar z)$. Fact \ref{F:CDM} (c) yields some constants $\ell_{D_1}$ and $\delta_{D_1}$ (only depending on $C$) for the formula $\rho'(v; \bar z, \bar u_1, \bar u_2)$ such that for every tuple $\bar c$ of a finite field $k$ and non-constant polynomials $P_1$ and $P_2$ over $k$ (modulo the above identification), the cardinality of the set $\rho'(v, \bar c, P_1, P_2)$ is either bounded by $\ell_{D_1}$ or is at least $\delta_{D_1}\cdot|k|$. Set now $t_C=\ell_{D_1}$. 

Assume for a contradiction that the statement of the theorem does not hold.  For a fixed complexity $C>0$ choose  for every natural number $n$ a finite field $k_n$ of characteristic at least $n$ as well as an $\LL_{\rm ring}$-definable subset $S_n$ of $k_n^2$, given by an instance $\varphi_n(x, y, c_n)$ for some $|\bar z|$-tuple $c_n$ in $k_n$ of a formula of $\varphi_n(x, y, \bar z)$ of complexity at most $C$ as well as two non-constant polynomials $P_{1,n}$ and $P_{2,n}$ of degree at most $C$ over $k_n$ such that the set $S_n$ has cardinality at least $\delta_C |k_n|^2$, where $\delta_C$ is as in the above paragraph, yet the set 
\[ 
\mathcal J_n=\left\{g \in k_n \ | \  |\mathcal X_g(k_n)|<\delta'_C|k|^3 \right\}
\] 
has size at least $t_C+1$, where $t_C=\ell_{D_1}$ and \[  \mathcal X_g(k_n) =\{ (x, y, y') \in k_n^3 \, | \, (x, y), (x, y'), (x+P_{1, n}(g), y) \text{ and } (x, y'+P_{2,n}(g)) \text{ lie in } S_n \}.\] 

Let $\mathcal U$ be a non-principal ultrafilter on $\N$. As in the proof of Theorem \ref{T:Roth_fte}, there is a fixed single formula $\varphi(x, y, \bar z)$ such that for $\mathcal U$-almost all $n$ in $\N$, the set $S_n$ is defined by $\varphi(x,y, \bar c_n)$. By construction and Corollary \ref{C:CDM_maxdim}, we have that $\bar c_n$ does not satisfy the associated formula $\rho(\bar z)$ to $\varphi$. 
Furthermore, Fact \ref{F:CDM} (c) and  Corollary \ref{C:CDM_maxdim} yields that the set 
\[
\mathcal J_n=\rho'(k_n, \bar c_n, P_{1,n},P_{2,n}).
\] 
has cardinality at least $\delta_{D_1}|k_n|$. 

Consider the pseudo-finite field $K=\prod_{n\to \mathcal U} k_n$ and set $\bar c=[(\bar c_n)_{n\in \N}]_{\mathcal U}$ in $K^{|\bar z|}$ as well as  the definable set $S=\varphi(K^2,\bar c)$. Let also  $P_1(\t)$ and $P_2(\t)$ be the polynomials over $K$ of degree at most $C$ given by the corresponding sequences $[(P_{1,n}(\t))_{n\in\N}]_{\mathcal U}$ and $[(P_{2,n}(\t))_{n\in \N}]_{\mathcal U}$. By \L o\'s's Theorem, the set $S$ has dimension $2$ by the above discussion and Corollary \ref{C:CDM_maxdim}, since $\bar c$ does not satisfy $\rho(\bar z)$. 

\L o\'s's Theorem again yields that $K$ has characteristic $0$ and hence is divisible, so $K$ equals its additive connected component by Fact \ref{F:G00} (b). For every non-constant polynomial $P(\t)$ with coefficients in the subfield $L$ of $K$, every element $b$ is interalgebraic with $P(b)$ over $L$. In particular, both (non-constant) polynomial functions $P_1(\t)$ and $P_2(\t)$ preserve generically the additive connected component. Moreover, if $g=[(g_n)_{n\in \N}]_{\mathcal U}$, the set $\mathcal X_g(K^3)=\prod_{n\to \mathcal U} \mathcal X_{g_n}(k^3_n)$.
By Proposition \ref{P:skew}, the set $\mathcal J=\{g \in K \ |  \ \dim(\mathcal X_g)<3\}$ is finite.  We will obtain the desired contradiction if we show that $\mathcal J=\prod_{n\to \mathcal U} \mathcal J_{n}$. Indeed, an element $g=[(g_n)_{n\in \N}]_{\mathcal U}$ belongs to $\mathcal J$ if and only if $g$ realizes the formula $\rho'(v, \bar c, P_1, P_2)$ in $K$, which by \L o\'s's Theorem  is equivalent to requiring that $g_n$ belongs to $\mathcal J_n$ for $\mathcal U$-almost all $n$, as desired.
\end{proof}

\end{document}